\documentclass{compositio}
\usepackage{mathrsfs}
\usepackage{amsmath}
\usepackage{amssymb}
\usepackage{verbatim}
\usepackage{xcolor}
\usepackage{graphicx}
\usepackage[all]{xy}
\numberwithin{equation}{section}

\theoremstyle{plain}
\newtheorem{theorem}{Theorem}[section]
\newtheorem{proposition}{Proposition}[section]
\newtheorem{corollary}{Corollary}[section]
\newtheorem{lemma}{Lemma}[section]

\theoremstyle{definition}
\newtheorem{definition}{Definition}[section]

\theoremstyle{remark}

\begin{document}

\title{A Nadel-type vanishing theorem concerning the asymptotic multiplier ideal sheaf}

\author{Jingcao Wu}
\email{jingcaowu13@fudan.edu.cn}
\address{School of Mathematical Sciences, Fudan University, Shanghai 200433, People's Republic of China}

\thanks{This research was supported by China Postdoctoral Science Foundation, grant 2019M661328.}

\begin{abstract}
In this paper we establish a Nadel-type vanishing theorem on a projective manifold $X$ concerning the asymptotic multiplier ideal sheaf.
\end{abstract}

\maketitle

\section{Introduction}
\label{sec:introduction}

The celebrated Nadel vanishing theorem says that
\begin{theorem}[(Nadel, \cite{Nad90})]
Let $X$ be a projective manifold of dimension $n$, and let $(L,\varphi)$ be a big line bundle on $X$. Then
\[
H^{q}(X,K_{X}\otimes L\otimes\mathscr{I}(\varphi))=0
\]
for $q>0$.
\end{theorem}

Here $\mathscr{I}(\varphi)$ refers to the multiplier ideal sheaf \cite{Nad90} associated to $\varphi$. This theorem can be seen as the analytic counterpart of the Kawamata--Viehweg vanishing theorem \cite{Kaw82,Vie82} in algebraic geometry, and has great applications. Then it is natural to generalise it to a pseudo-effective line bundle. One could refer to \cite{Cao14,Eno93,Hua19,Mat14,Mat15a,Mat15b,Mat18,WaZ19} and the references therein for several generalisations. In practice, we find that the lower bound of the order $q$ such that $H^{q}(X,K_{X}\otimes L\otimes\mathscr{I}(\varphi))=0$ usually depends on the numerical dimension $\textrm{nd}(L)$ or Iitaka dimension $\kappa(L)$ \cite{Laz04} of $L$. For example, in Theorem 1.1, $\kappa(L)=\textrm{nd}(L)=n$, and the lower bound is just $0=n-\kappa(L)$.

In this paper, we present the following Nadel-type vanishing theorem concerning the asymptotic multiplier ideal sheaf $\mathscr{I}(\|L\|)$ (see Sect.\ref{sec:asymptotic}).
\begin{theorem}\label{t11}
Let $X$ be a projective manifold of dimension $n$, and let $L$ be a pseudo-effective line bundle. Then we have
\[
   H^{q}(X,K_{X}\otimes L\otimes\mathscr{I}(\|L\|))=0
\]
for $q>n-\kappa(L)$.
\end{theorem}

In particular, if $L$ is nef and abundant (see Sect.\ref{sec:abundant}), we have
\begin{corollary}\label{c11}
Let $X$ be a projective manifold of dimension $n$, and let $L$ be a nef and abundant line bundle. Then we have
\[
   H^{q}(X,K_{X}\otimes L)=0
\]
for $q>n-\textrm{nd}(L)$.
\end{corollary}

Theorem \ref{t11} can be seen as a generalisation of the original Nadel vanishing theorem in the view of the following variant \cite{Laz04b} of Theorem 1.1.
\begin{theorem}[(=Theorem 11.2.12, \cite{Laz04b})]
Let $X$ be a projective manifold of dimension $n$, and let $L$ be a big line bundle on $X$. Then
\[
H^{q}(X,K_{X}\otimes L\otimes\mathscr{I}(\|L\|))=0
\]
for $q>0$.
\end{theorem}
On the other hand, when we deal with the asymptotic multiplier ideal sheaf instead of the multiplier ideal sheaf, the presentation of the vanishing result for a pseudo-effective line bundle is considerably simplified comparing with \cite{Cao14,Mat14,Mat15a,Mat15b}.

The proof of Theorem \ref{t11} uses the same strategy as \cite{Mat14}. We first prove an injectivity theorem and an asymptotic estimate for the order of the cohomology group as follows:

\begin{theorem}[(Injectivity theorem)]\label{t12}
Let $X$ be a compact K\"{a}hler manifold of dimension $n$. Let $L$ and $H$ be line bundles on $X$ with $\kappa(L)\geqslant0$ and $\kappa(H)\geqslant0$. Let $\varphi_{L}$ and $\varphi_{L\otimes H}$ be the singular metrics on $L$ and $L\otimes H$, which is associated to $\mathscr{I}(\|L\|)$ and $\mathscr{I}(\|L\otimes H\|)$ respectively (see Sect.\ref{sec:asymptotic}). Assume that $i\Theta_{L,\varphi_{L}}\geqslant\delta i\Theta_{L\otimes H,\varphi_{L\otimes H}}$ for some positive number $\delta$.

For a (non-zero) section $s$ of $H$, the multiplication map induced by the tensor product with $s$
\[
\Phi:H^{q}(X,K_{X}\otimes L\otimes\mathscr{I}(\|L\|))\rightarrow H^{q}(X,K_{X}\otimes L\otimes H\otimes\mathscr{I}(\|L\otimes H\|))
\]
is well-defined and injective for any $q\geqslant0$.
\end{theorem}
Notice that the assumptions $\kappa(L)\geqslant0$ and $\kappa(H)\geqslant0$ are required in order to define $\mathscr{I}(\|L\|)$ and $\mathscr{I}(\|H\|)$. To my best acknowledgement, Theorem \ref{t12} cannot be obtained by directly applying the available injectivity theorems, such as those in \cite{Eno93,Fuj12,Ko86a,Mat14,Mat15a,Mat15b,Mat18}. The basic reason is that if $\varphi_{H}$ is the singular metric associated to $\mathscr{I}(\|H\|)$, in general we do not have $\varphi_{L}+\varphi_{H}=\varphi_{L\otimes H}$.

For any coherent sheaf $\mathcal{F}$, let $h^{q}(\mathcal{F})$ be the dimension of $H^{q}(X,\mathcal{F})$. Let $L^{k}$ be the $k$-th tensor product of a line bundle $L$. Then we have
\begin{theorem}[(Asymptotic estimate)]\label{t13}
Let $X$ be a projective manifold of dimension $n$, and let $L$ be a pseudo-effective line bundle on $X$. Then for any coherent sheaf $\mathcal{G}$ and $q\geqslant0$, we have
\[
    h^{q}(L^{k}\otimes\mathcal{G}\otimes\mathscr{I}(\|L^{k}\|))=O(k^{n-q}).
\]
\end{theorem}

Combining with Theorems \ref{t12} and \ref{t13}, we then finish the proof of Theorem \ref{t11}. The details are provided in the text.

Eventually, we present a relative version of Theorem \ref{t11}.
\begin{theorem}\label{t15}
Let $f:X\rightarrow Y$ be a surjective morphism between projective manifolds, and let $L$ be a pseudo-effective line bundle on $X$. Let $l$ be the dimension of a general fibre $F$ of $f$. Then
\[
R^{q}f_{\ast}(K_{X}\otimes L\otimes\mathscr{I}(f,\|L\|))=0
\]
for $q>l-\kappa(L,f)$. Here $\mathscr{I}(f,\|L\|)$ is the relative version of the asymptotic multiplier ideal sheaf (see Sect.\ref{sec:asymptotic}) and $\kappa(L,f)$ is the relative Iitaka dimension (see Sect.\ref{sec:iitaka}).
\end{theorem}
We remark here that Theorem \ref{t15} cannot be easily obtained by applying Theorems \ref{t12} and \ref{t13} on the general fibre.

The plan of this paper is as follows. In Sect.\ref{sec:preliminary} we give a brief introduction on all the background materials including the asymptotic multiplier ideal sheaf, abundant line bundle and so on. In Sect.\ref{sec:harmonic} we develop the harmonic theory and prove Theorem \ref{t12}. In Sect.\ref{sec:estimate} we prove the asymptotic estimate, i.e. Theorem \ref{t13}. Then we prove Theorem \ref{t11} in next section. In the final section we extend everything to the relative setting and prove Theorem \ref{t15}.

\section{Preliminary}
\label{sec:preliminary}
In this section we will introduce some basic materials. For clarity and for convenience of later reference, it will be done in a general setting, i.e. on a K\"{a}hler manifold.

Let $(X,\omega)$ be a compact K\"{a}hler manifold, and let $L$ be a pseudo-effective line bundle on $X$.

\subsection{The asymptotic multiplier ideal sheaf}
\label{sec:asymptotic}
This part is mostly collected from \cite{Laz04b}.

Recall that for an arbitrary ideal sheaf $\mathfrak{a}\subset\mathcal{O}_{X}$, the associated multiplier ideal sheaf is defined as follows: let $\mu:\tilde{X}\rightarrow X$ be a smooth modification such that $\mu^{\ast}\mathfrak{a}=\mathcal{O}_{\tilde{X}}(-E)$, where $E$ has simple normal crossing support. Then given a positive real number $c>0$ the multiplier ideal sheaf is defined as
\[
\mathscr{I}(c\cdot\mathfrak{a}):=\mu_{\ast}\mathcal{O}_{\tilde{X}}(K_{\tilde{X}/X}-\lfloor cE\rfloor).
\]
Here $K_{\tilde{X}/X}$ is the relative canonical divisor and $\lfloor E\rfloor$ means the round-down.

Now assume that $\kappa(L)\geqslant0$. Fix a positive real number $c>0$. For $k>0$ consider the complete linear series $|L^{k}|$, and form the multiplier ideal sheaf
\[
\mathscr{I}(\frac{c}{k}|L^{k}|)\subseteq\mathcal{O}_{X},
\]
where $\mathscr{I}(\frac{c}{k}|L^{k}|):=\mathscr{I}(\frac{c}{k}\cdot\mathfrak{a}_{k})$ with $\mathfrak{a}_{k}$ being the base-ideal of $|L^{k}|$. It is not hard to verify that for every integer $p\geqslant1$ one has the inclusion
\[
\mathscr{I}(\frac{c}{k}|L^{k}|)\subseteq\mathscr{I}(\frac{c}{pk}|L^{pk}|).
\]
Therefore the family of ideals
\[
\{\mathscr{I}(\frac{c}{k}|L^{k}|)\}_{(k\geqslant0)}
\]
has a unique maximal element from the ascending chain condition on ideals.

\begin{definition}\label{d21}
The asymptotic multiplier ideal sheaf associated to $c$ and $|L|$,
\[
\mathscr{I}(c\|L\|)
\]
is defined to be the unique maximal member among the family of ideals $\{\mathscr{I}(\frac{c}{k}|L^{k}|)\}$.
\end{definition}

By definition, $\mathscr{I}(c\|L\|)=\mathscr{I}(\frac{c}{k}|L^{k}|)$ for some $k$. Let $u_{1},...,u_{m}$ be a basis of $H^{0}(X,L^{k})$, then the base-ideal of $|L^{k}|$ is just $\mathcal{I}(u_{1},...,u_{m})$. Let $\varphi=\log(|u_{1}|^{2}+\cdots+|u_{m}|^{2})$, which is a singular metric on $L^{k}$. We verify that
\[
\mathscr{I}(\frac{c}{k}|L^{k}|)=\mathscr{I}(\frac{c}{k}\varphi).
\]
Indeed, let $\mu:\tilde{X}\rightarrow X$ be the smooth modification such that $\mu^{\ast}\mathcal{I}(u_{1},...,u_{m})=\mathcal{O}_{\tilde{X}}(-E)$, where $E$ has simple normal crossing support. Then it is computed in \cite{Dem12} that
\[
\mathscr{I}(\frac{c}{k}\varphi)=\mu_{\ast}\mathcal{O}_{\tilde{X}}(K_{\tilde{X}/X}-\lfloor \frac{c}{k}E\rfloor),
\]
which coincides with the definition of $\mathscr{I}(\frac{c}{k}|L^{k}|)$. In summary, we have
\[
\mathscr{I}(c\|L\|)=\mathscr{I}(\frac{c}{k}\varphi),
\]
and $\frac{1}{k}\varphi$ is called the singular metric on $L$ associated to $\mathscr{I}(\|L\|)$.

Next, we introduce the relative variant. Let $f:X\rightarrow Y$ be a surjective morphism between projective manifolds, and $L$ a line bundle on $X$ whose restriction to a general fibre of $f$ has non-negative Iitaka dimension. Then there is a naturally defined homomorphism
\[
\rho:f^{\ast}f_{\ast}L\rightarrow L.
\]
Let $\mu:\tilde{X}\rightarrow X$ be a smooth modification of $|L|$ with respect to $f$, having the property that the image of the induced homomorphism
\[
\mu^{\ast}\rho:\mu^{\ast}f^{\ast}f_{\ast}L\rightarrow \mu^{\ast}L
\]
is the subsheaf $\mu^{\ast}L\otimes\mathcal{O}_{\tilde{X}}(-E)$ of $\mu^{\ast}L$, $E$ being an effective divisor on $\tilde{X}$ such that $E+\textrm{except}(\mu)$ has simple normal crossing support. Here $\textrm{except}(\mu)$ is the exceptional divisor of $\mu$. Given $c>0$ we define
\[
\mathscr{I}(f,c|L|)=\mu_{\ast}\mathcal{O}_{\tilde{X}}(K_{\tilde{X}/X}-\lfloor cE\rfloor).
\]
Similarly, $\{\mathscr{I}(f,\frac{c}{k}|L^{k}|)\}_{(k\geqslant0)}$ has a unique maximal element.

\begin{definition}\label{d22}
The relative asymptotic multiplier ideal sheaf associated to $f$, $c$ and $|L|$,
\[
\mathscr{I}(f,c\|L\|)
\]
is defined to be the unique maximal member among the family of ideals $\{\mathscr{I}(f,\frac{c}{k}|L^{k}|)\}$.
\end{definition}

By definition, $\mathscr{I}(f,c\|L\|)=\mathscr{I}(f,\frac{c}{k}|L^{k}|)$ for some $k$. Let $\rho$ be the naturally defined homomorphism
\[
\rho:f^{\ast}f_{\ast}L^{k}\rightarrow L^{k}
\]
by abusing the notation. Let $\mu:\tilde{X}\rightarrow X$ be the smooth modification of $|L^{k}|$ with respect to $f$ such that $\textrm{Im}(\mu^{\ast}\rho)=\mu^{\ast}L^{k}\otimes\mathcal{O}_{\tilde{X}}(-E)$. Consider $\mu_{\ast}\mathcal{O}_{\tilde{X}}(-E)$ which is an ideal sheaf on $X$. Pick a local coordinate ball $U$ of $Y$, and let $u_{1},...,u_{m}$ be the generators of $\mu_{\ast}\mathcal{O}_{\tilde{X}}(-E)$ on $f^{-1}(U)$. The existence of these generators is obvious concerning the fact that $\textrm{Im}(\mu^{\ast}\rho)=\mu^{\ast}L^{k}\otimes\mathcal{O}_{\tilde{X}}(-E)$. Moreover they can be seen as the sections of $\Gamma(f^{-1}(U),L^{k})$.

Now let $\varphi_{U}=\log(|u_{1}|^{2}+\cdots+|u_{m}|^{2})$, which is a singular metric on $L^{k}|_{f^{-1}(U)}$. It is then easy to verify that
\[
\mathscr{I}(\frac{c}{k}\varphi_{U})=\mathscr{I}(f,\frac{c}{k}|L^{k}|)\textrm{ on }f^{-1}(U).
\]
Furthermore, if $v_{1},...,v_{m}$ are alternative generators and $\psi_{U}=\log(|v_{1}|^{2}+\cdots+|v_{m}|^{2})$, obviously we have $\mathscr{I}(\frac{c}{k}\varphi_{U})=\mathscr{I}(\frac{c}{k}\psi_{U})$. Hence all the $\mathscr{I}(\frac{c}{k}\varphi_{U})$ patch together to give a globally defined multiplier ideal sheaf $\mathscr{I}(\frac{c}{k}\varphi)$ such that
\[
\mathscr{I}(\frac{c}{k}\varphi)=\mathscr{I}(f,\frac{c}{k}|L^{k}|)=\mathscr{I}(f,c\|L\|)\textrm{ on }X.
\]
One should be careful that $\{\frac{1}{k}\varphi_{U}\}$ won't give a globally defined metric on $L$ in general. Hence $\frac{1}{k}\varphi$ is interpreted as the collection of functions $\{\frac{1}{k}\varphi_{U}\}$ by abusing the notation, which is called the (local) singular metric on $L$ associated to $\mathscr{I}(f,c\|L\|)$.

Now we collect some elementary properties from \cite{Laz04b}. Recall that for an ideal sheaf $\mathcal{I}$ on $X$, the corresponding ideal sheaf relative to $f$ is defined as
\[
\mathcal{I}_{f}:=\textrm{Im}(f^{\ast}f_{\ast}\mathcal{I}\rightarrow\mathcal{O}_{X}).
\]
\begin{proposition}\label{p21}
Let $f:X\rightarrow Y$ be a surjective morphism between compact K\"{a}hler manifolds, and $H_{1}, H_{2}$ are line bundles on $X$ whose restriction to a general fibre of $f$ has non-negative Iitaka dimension. Let $L_{1}, L_{2}$ be line bundles on $X$ with non-negative Iitaka dimension. $m$ and $k$ are non-negative integers.
\begin{enumerate}
 \item The natural inclusion
 \[
 H^{0}(X,L^{k}\otimes\mathscr{I}(\|L^{k}_{1}\|))\rightarrow H^{0}(X,L^{k}_{1})
 \]
 is an isomorphism for every $k\geqslant1$.
 \item Let $\mathfrak{a}_{m}=\mathfrak{a}(|L^{m}_{1}|)$ be the base-ideal of $|L^{m}_{1}|$,  where by convention we set $\mathfrak{a}_{m}=(0)$ if $|L^{m}_{1}|=\emptyset$. Then
\[
\mathfrak{a}_{m}\cdot\mathscr{I}(\|L^{k}_{2}\|)\subseteq\mathscr{I}(\|L^{m}_{1}\otimes L^{k}_{2}\|).
\]
 \item $\mathscr{I}(\|L^{k}_{1}\|)\supseteq\mathscr{I}(\|L^{k+1}_{1}\|)$ for every $k$.
\item Let $\mathfrak{a}_{k,f}=\mathfrak{a}(f,|H^{k}_{1}|)$ be the base-ideal of $|H^{k}_{1}|$ relative to $f$. There exits a integer $k_{0}$ such that for every $k\geqslant k_{0}$, the canonical map $\rho_{k}:f^{\ast}f_{\ast}H^{k}_{1}\rightarrow H^{k}_{1}$ factors through the inclusion $H^{k}_{1}\otimes\mathscr{I}(f,\|H^{k}_{1}\|)$, i.e.
\[
\mathfrak{a}_{k,f}\subseteq\mathscr{I}(f,\|H^{k}_{1}\|).
\]
Equivalently, the natural map
\[
f_{\ast}(H^{k}_{1}\otimes\mathscr{I}(f,\|H^{k}_{1}\|))\rightarrow f_{\ast}(H^{k}_{1})
\]
is an isomorphism.
\item $\mathfrak{a}_{m,f}\cdot\mathscr{I}(f,\|H^{k}_{2}\|)\subseteq\mathscr{I}(f,\|H^{m}_{1}\otimes H^{k}_{2}\|)$.
\item $\mathscr{I}(f,\|H^{k}_{1}\|)\supseteq\mathscr{I}(f,\|H^{k+1}_{1}\|)$ for every $k$.
\end{enumerate}
\begin{proof}
(i) is proved in \cite{Laz04b}, Proposition 11.2.10.

(ii) Fix $p\gg0$ and divisible enough that computes all of the multiplier ideals $\mathscr{I}(\|L^{m}_{1}\|)$, $\mathscr{I}(\|L^{k}_{2}\|)$ and $\mathscr{I}(\|L^{m}_{1}\otimes L^{k}_{2}\|)$. Let $\mathfrak{b}_{k}$ be the base-ideal of $|L^{k}_{2}|$, and let $\mathfrak{c}_{m,k}$ be the base-ideal of $|L^{m}_{1}\otimes L^{k}_{2}|$. Let $\mu:\tilde{X}\rightarrow X$ be the smooth modification of $\mathfrak{a}_{m}$, $\mathfrak{a}_{pm}$ $\mathfrak{b}_{pk}$ and $\mathfrak{c}_{pm,pk}$, such that
\[
 \mu^{\ast}\mathfrak{a}_{m}=\mathcal{O}_{\tilde{X}}(-E), \mu^{\ast}\mathfrak{a}_{pm}=\mathcal{O}_{\tilde{X}}(-F), \mu^{\ast}\mathfrak{b}_{pk}=\mathcal{O}_{\tilde{X}}(-G) \textrm{ and } \mu^{\ast}\mathfrak{c}_{pm,pk}=\mathcal{O}_{\tilde{X}}(-H),
\]
where $E=\sum a_{i}E_{i}$, $F=\sum b_{i}E_{i}$, $G=\sum c_{i}E_{i}$ and $H=\sum d_{i}E_{i}$ have simple normal crossing support. Then for every $i$,
\[
d_{i}\leqslant b_{i}+c_{i}\leqslant pa_{i}+c_{i}
\]
and consequently
\[
-a_{i}-\lfloor\frac{c_{i}}{p}\rfloor\leqslant -\lfloor\frac{d_{i}}{p}\rfloor.
\]
Thus
\[
\begin{split}
\mathfrak{a}_{m}\cdot\mathscr{I}(\|L^{k}_{2}\|)&\subseteq\mu_{\ast}\mathcal{O}_{\tilde{X}}(-E+K_{\tilde{X}/X}-\lfloor\frac{1}{p}G)\rfloor)\\
&\subseteq\mu_{\ast}\mathcal{O}_{\tilde{X}}(K_{\tilde{X}/X}-\lfloor\frac{1}{p}H)\rfloor)\\
&=\mathscr{I}(\|L^{m}_{1}\otimes L^{k}_{2}\|).
\end{split}
\]

(iii) is proved in \cite{Laz04b}, Proposition 11.1.8.

(iv) is proved in \cite{Laz04b}, Proposition 11.2.15.

(v) is similar with (ii), and we omit it here.

(vi) Fix $p\gg0$ and divisible enough that computes both of the multiplier ideals $\mathscr{I}(f,k\|H_{1}\|)$ and $\mathscr{I}(f,\|H^{k}_{1}\|)$. Then
\[
\begin{split}
\mathscr{I}(f,\|H^{k}_{1}\|)&=\mathscr{I}(f,\frac{1}{p}|H^{pk}_{1}|)\\
&=\mathscr{I}(f,\frac{k}{pk}|H^{pk}_{1}|)\\
&=\mathscr{I}(f,k\|H_{1}\|).
\end{split}
\]
Now we have
\[
\begin{split}
\mathscr{I}(f,\|H^{k}_{1}\|)&=\mathscr{I}(f,k\|H_{1}\|)\\
&\supseteq\mathscr{I}(f,(k+1)\|H_{1}\|)\\
&=\mathscr{I}(f,\|H^{k+1}_{1}\|).
\end{split}
\]
\end{proof}
\end{proposition}

\subsection{Abundant line bundle}
\label{sec:abundant}
\begin{definition}\label{d23}
A line bundle $L$ is said to be abundant if $\kappa(L)=\textrm{nd}(L)$.
\end{definition}

This notion arises naturally. Moreover, A nef and abundant line bundle can be characterised by asymptotic multiplier ideal sheaf as follows.
\begin{theorem}[(Russo, \cite{Rus09})]
Assume that $\kappa(L)\geqslant0$. Then
\[
\mathscr{I}(\|L^{k}\|)=\mathcal{O}_{X}
\]
for all $k$ if and only if $L$ is nef and abundant.
\end{theorem}

Corollary \ref{c11} now follows immediately from Theorems \ref{t11} and 2.1.

\subsection{Relative Iitaka dimension}
\label{sec:iitaka}
Let $f:X\rightarrow Y$ be a surjective morphism between projective manifolds, and $L$ a line bundle on $X$. Let $l$ be the dimension of a general fibre $F$ of $f$. We have
\begin{proposition}\label{p22}
For every coherent sheaf $\mathcal{G}$ on $X$, there is $C>0$ (independent of $L$) such that
\[
\textrm{rank}(f_{\ast}(\mathcal{G}\otimes L^{k}))\leqslant Ck^{l} \textrm{ for all }k\gg0.
\]
\begin{proof}
Let us write $L=A\otimes B^{-1}$, with $A$ and $B$ are very ample line bundles. For every $k$, if we choose $E$ general in the complete linear system $|B^{k}|$, then a
local defining function of $E$ is a non-zero divisor on $\mathcal{G}$, in which case we have an inclusion
\[
f_{\ast}(\mathcal{G}\otimes L^{k})\hookrightarrow f_{\ast}(\mathcal{G}\otimes A^{k}).
\]
Since $A$ is very ample, we know that there is a polynomial $P\in\mathbb{Q}[t]$ \cite{Laz04} with $\deg(P)\leqslant l$ such that $h^{0}(F,\mathcal{G}\otimes A^{k})=P(k)$ for $k\gg0$. Therefore $h^{0}(F,\mathcal{G}\otimes L^{k})\leqslant P(k)\leqslant Ck^{l}$ for a suitable $C>0$ and all $k\gg0$.
\end{proof}
\end{proposition}

\begin{definition}\label{d24}
The relative Iitaka dimension $\kappa(L,f)$ of $L$ is the biggest integer $m$ such that there is $C>0$ satisfying
\[
\textrm{rank}(f_{\ast}L^{k})\geqslant Ck^{m} \textrm{ for all }k\gg0
\]
with the convention that $\kappa(L,f)=-\infty$ if $\textrm{rank }f_{\ast}L^{k}=0$.
\end{definition}
Note that $\kappa(L,f)$ takes value in $\{-\infty, 0,1,...,l\}$ by Proposition \ref{p22}. In particular, if $\kappa(L,f)=l$, we say that $L$ is $f$-big.

\section{The harmonic theory}
\label{sec:harmonic}
Let $(L,\varphi)$ be a pseudo-effective line bundle on a compact K\"{a}hler manifold $(X,\omega)$. Assume that there exit integers $k_{0}$, $m$ and sections $s_{1},...,s_{m}\in L^{k_{0}}$ such that
\[
(|s_{1}|^{2}+\cdots+|s_{m}|^{2})e^{-k_{0}\varphi}
\]
is bounded on $X$. In this section, we will develop the harmonic theory on such a line bundle.

\subsection{The harmonic forms}
\label{sec:forms}

The Laplacian operator associated to a singular metric $\varphi$ is not well-defined in canonical harmonic theory. Fortunately, by Demailly's approximation technique \cite{DPS01}, we can find a family of metrics $\{\varphi_{\varepsilon}\}$ on $L$ with the following properties:

(a) $\varphi_{\varepsilon}$ is smooth on $X-Z_{\varepsilon}$ for a closed subvariety $Z_{\varepsilon}$;

(b) $\varphi\leqslant\varphi_{\varepsilon_{1}}\leqslant\varphi_{\varepsilon_{2}}$ holds for any $0<\varepsilon_{1}\leqslant\varepsilon_{2}$;

(c) $\mathscr{I}(\varphi)=\mathscr{I}(\varphi_{\varepsilon})$; and

(d) $i\Theta_{L,\varphi_{\varepsilon}}\geqslant-\varepsilon\omega$.

Since
\[
(|s_{1}|^{2}+\cdots+|s_{m}|^{2})e^{-k_{0}\varphi}
\]
is bounded on $X$, the pole-set of $\varphi_{\varepsilon}$ for every $\varepsilon>0$ is contained in the subvariety
\[
Z:=\{x|s_{1}(x)=\cdots=s_{m}(x)=0\}
\]
by property (b). Hence, instead of (a), we can assume that

(a') $\varphi_{\varepsilon}$ is smooth on $X-Z$, where $Z$ is a closed subvariety of $X$ independent of $\varepsilon$.

Throughout this paper, when saying that $\{\psi_{\varepsilon}\}$ is a regularising sequence of a singular metric $\psi$, we always refer to such a family of metrics with properties (a'), (b), (c) and (d).

Now let $Y=X-Z$. We use the method in \cite{Dem82} to construct a complete K\"{a}hler metric on $Y$ as follows. Since $Y$ is weakly pseudo-convex, we can take a smooth plurisubharmonic exhaustion function $\psi$ on $X$. Define $\tilde{\omega}_{l}=\omega+\frac{1}{l}i\partial\bar{\partial}\psi^{2}$ for $l\gg0$. It is easy to verify that $\tilde{\omega}_{l}$ is a complete K\"{a}hler metric on $Y$ and $\tilde{\omega}_{l_{1}}\geqslant\tilde{\omega}_{l_{2}}\geqslant\omega$ for $l_{1}\leqslant l_{2}$.

Let $L^{n,q}_{(2)}(Y,L)_{\varphi_{\varepsilon},\tilde{\omega}_{l}}$ be the $L^{2}$-space of $L$-valued $(n,q)$-forms on $Y$ with respect to the inner product given by $\varphi_{\varepsilon},\tilde{\omega}$. Then we have the orthogonal decomposition
\begin{equation}\label{e31}
L^{n,q}_{(2)}(Y,L)_{\varphi_{\varepsilon},\tilde{\omega}_{l}}=\mathrm{Im}\bar{\partial}\bigoplus\mathcal{H}^{n,q}_{\varphi_{\varepsilon}, \tilde{\omega}_{l}}(L)\bigoplus\mathrm{Im}\bar{\partial}^{\ast}_{\varphi_{\varepsilon}}
\end{equation}
where
\[
\begin{split}
\mathrm{Im}\bar{\partial}&=\mathrm{Im}(\bar{\partial}:L^{n,q-1}_{(2)}(Y,L)_{\varphi_{\varepsilon},\tilde{\omega}_{l}}\rightarrow L^{n,q}_{(2)}(Y,L)_{\varphi_{\varepsilon},\tilde{\omega}_{l}}),\\
  \mathcal{H}^{n,q}_{\varphi_{\varepsilon}, \tilde{\omega}_{l}}(L)&=\{\alpha\in L^{n,q}_{(2)}(Y,L)_{\varphi_{\varepsilon},\tilde{\omega}_{l}};\bar{\partial}\alpha=0, \bar{\partial}^{\ast}_{\varphi_{\varepsilon}}\alpha=0\},
\end{split}
\]
and
\[
\mathrm{Im}\bar{\partial}^{\ast}_{\varphi_{\varepsilon}}=\mathrm{Im}(\bar{\partial}^{\ast}_{\varphi_{\varepsilon}}:L^{n,q+1}_{(2)}(Y,L)_{\varphi_{\varepsilon},\tilde{\omega}_{l}}\rightarrow L^{n,q}_{(2)}(Y,L)_{\varphi_{\varepsilon},\tilde{\omega}_{l}}).
\]
We give a brief explanation for the decomposition (\ref{e31}). Usually $\mathrm{Im}\bar{\partial}$ is not closed in the $L^{2}$-space of a noncompact manifold even if the metric is complete. However, in the situation we consider here, $Y$ has the compactification $X$, and the forms on $Y$ are bounded in $L^{2}$-norms. Such a form will have good extension properties. Therefore the set $L^{n,q}_{(2)}(Y,L)_{\varphi_{\varepsilon},\tilde{\omega}_{l}}\cap\mathrm{Im}\bar{\partial}$ behaves much like the space
\[
\mathrm{Im}(\bar{\partial}:L^{n,q-1}_{(2)}(X,L)_{\varphi,\omega}\rightarrow L^{n,q}_{(2)}(X,L)_{\varphi,\omega}),
\]
which is surely closed. The complete explanation can be found in \cite{Fuj12,Wu17}.

Now we have all the ingredients for the definition of $\Box_{0}$-harmonic forms. We denote the Lapalcian operator on $Y$ associated to $\tilde{\omega}_{l}$ and $\varphi_{\varepsilon}$ by $\Box_{\varepsilon,l}$. Recall that for two $L$-valued $(n,q)$-forms $\alpha,\beta$ (not necessary to be $\bar{\partial}$-closed), we say that they are cohomologically equivalent if there exits an $L$-valued $(n,q-1)$-form $\gamma$ such that $\alpha=\beta+\bar{\partial}\gamma$. We denote by $\alpha\in[\beta]$ this equivalence relationship.
\begin{definition}\label{d31}
Let $\alpha$ be an $L$-valued $(n,q)$-form on $X$ with bounded $L^{2}$-norm with respect to $\omega,\varphi$. Assume that for every $\varepsilon\ll1$ and $l\gg1$, there exists a Dolbeault cohomological equivalent class $\alpha_{\varepsilon,l}\in[\alpha|_{Y}]$ such that
\begin{enumerate}
  \item $\Box_{\varepsilon,l}\alpha_{\varepsilon,l}=0$ on $Y$;
  \item $\alpha_{\varepsilon,l}\rightarrow\alpha|_{Y}$ in $L^{2}$-norm.
\end{enumerate}

Then we call $\alpha$ a $\Box_{0}$-harmonic form. The space of all the $\Box_{0}$-harmonic forms is denoted by
\[
\mathcal{H}^{n,q}(X,L\otimes\mathscr{I}(\varphi)).
\]
\end{definition}

\subsection{The Hodge-type isomorphism}
\label{sec:hodge}

Firstly, we prove a regularity result concerning the $\Box_{0}$ operator.
\begin{proposition}\label{p31}
Let $\alpha$ be an $L$-valued $(n,q)$-form (not necessary to be smooth) on $X$ whose $L^{2}$-norm against $\varphi$ is bounded. Then
\begin{enumerate}
  \item if $\alpha$ is $\Box_{0}$-harmonic, $\bar{\partial}(\ast\alpha)=0$. Equivalently, $\ast\alpha$ is holomorphic.
  \item if $\Box_{0}\alpha=0$, $\alpha$ must be smooth.
\end{enumerate}
\begin{proof}
(i) We denote $\alpha|_{Y}$ simply by $\alpha_{Y}$. Since $\alpha$ is $\Box_{0}$-harmonic, there exists an $\alpha_{l,\varepsilon}\in[\alpha_{Y}]$ with $\alpha_{l,\varepsilon}\in\mathcal{H}^{n,q}_{\varphi_{\varepsilon},\tilde{\omega}_{l}}(L)$ for every $l,\varepsilon$ such that $\lim\alpha_{l,\varepsilon}=\alpha_{Y}$. In particular, $\bar{\partial}\alpha_{l,\varepsilon}=\bar{\partial}^{\ast}_{\varphi_{\varepsilon}}\alpha_{l,\varepsilon}=0$.

Let's recall the generalized Kodaira--Akizuki--Nakano formula in \cite{Tak95}. Let $\psi$ be a smooth real-valued function on $X$, and let $\chi$ be a smooth metric on $L$. Then we have
\begin{equation}\label{e32}
\begin{split}
&\|\sqrt{\eta}(\bar{\partial}+\bar{\partial}\psi\wedge)\alpha\|^{2}_{\chi}+\|\sqrt{\eta}\bar{\partial}^{\ast}\alpha\|^{2}_{\chi}\\
=&\|\sqrt{\eta}(\partial_{\chi}-\partial\psi\wedge)^{\ast}\alpha\|^{2}_{\chi}+\|\sqrt{\eta}\partial_{\chi}\alpha\|^{2}_{\chi}+<i\eta[\Theta_{L,\chi}+\partial\bar {\partial}\psi,\Lambda]\alpha,\alpha>_{\chi}
\end{split}
\end{equation}
for any $\alpha\in A^{p,q}(X,L)$ and $\eta=e^{\psi}$. Here $\partial_{\chi}$ is the $(1,0)$-part of the Chern connection associated to $\chi$. We remark here that formula (\ref{e32}) is also valid on $Y$ since the real codimension of $Y$ is at least $2$.

Apply formula (\ref{e32}) on $Y$ with $\eta=1$, we have
\begin{equation}\label{e33}
\begin{split}
0=&\|\bar{\partial}\alpha_{l,\varepsilon}\|^{2}_{\varphi_{\varepsilon},\tilde{\omega}_{l}}+\|\bar{\partial}^{\ast}_{\varphi_{\varepsilon}}\alpha_{l,\varepsilon}\|^{2}_{\varphi_{\varepsilon},\tilde{\omega}_{l}}\\
=&\|\partial^{\ast}_{\varphi_{\varepsilon}}\alpha_{l,\varepsilon}\|^{2}_{\varphi_{\varepsilon},\tilde{\omega}_{l}}+<i[\Theta_{L,\varphi_{\varepsilon}},\Lambda]\alpha_{l,\varepsilon},\alpha_{l,\varepsilon}>_{\varphi_{\varepsilon},\tilde{\omega}_{l}}.
\end{split}
\end{equation}
Remember that $i\Theta_{L,\varphi_{\varepsilon}}\geqslant-\varepsilon\omega$, thus
\[
<i[\Theta_{L,\varphi_{\varepsilon}},\Lambda]\alpha_{l,\varepsilon},\alpha_{l,\varepsilon}>_{\varphi_{\varepsilon},\tilde{\omega}_{l}}\geqslant-\varepsilon^{\prime}\tilde{\omega}_{l}
\]
by elementary computation. In particular, $\varepsilon^{\prime}\rightarrow0$ as $\varepsilon$ tends zero and $l$ tends infinity. Now take the limit on the both sides of formula (\ref{e33}) with respect to $l,\varepsilon$, we eventually obtain that
\[
\begin{split}
\lim\|\partial^{\ast}_{\varphi_{\varepsilon}}\alpha_{l,\varepsilon}\|^{2}_{\varphi_{\varepsilon},\tilde{\omega}_{l}}=\lim<i[\Theta_{L,\varphi_{\varepsilon}},\Lambda]\alpha_{l,\varepsilon},\alpha_{l,\varepsilon}>_{\varphi_{\varepsilon},\tilde{\omega}_{l}}=0.
\end{split}
\]
In particular,
\[
0=\lim\partial^{\ast}_{\varphi_{\varepsilon}}\alpha_{l,\varepsilon}=\ast\bar{\partial}\ast\lim\alpha_{l,\varepsilon}=\ast\bar{\partial}\ast\alpha
\]
in $L^{2}$-topology. Equivalently, $\bar{\partial}\ast\alpha=0$ on $Y$ in analytic topology. Hence $\ast\alpha$ is a holomorphic $L$-valued $(n-q,0)$-form on $Y$. On the other hand, since $\ast\alpha$ has the bounded $L^{2}$-norm on $Y$, it extends to the whole space by classic $L^{2}$-extension theorem \cite{Ohs02}. The extension is still denoted by $\ast\alpha$, which is an $L$-valued holomorphic $(n-q,0)$-form on $X$.

(ii) Since $\alpha=c_{n-q}\frac{\omega^{q}}{q!}\wedge\ast\alpha$, $\alpha$ must be smooth.
\end{proof}
\end{proposition}

Next, we generalise Hodge's theorem to show that Definition \ref{d31} is meaningful.
\begin{proposition}\label{p32}
Let $(X,\omega)$ be a compact K\"{a}hler manifold. $(L,\varphi)$ is a pseudo-effective line bundle on $X$, and $E$ is an arbitrary vector bundle. Then the following isomorphism holds for all $k$:
\begin{equation}\label{e34}
\begin{split}
   \mathcal{H}^{n,q}(X,L^{k}\otimes E\otimes\mathscr{I}(k\varphi))&\simeq H^{n,q}(X,L^{k}\otimes E\otimes\mathscr{I}(k\varphi))
\end{split}
\end{equation}
In particular, when $\varphi$ is smooth, $\alpha\in\mathcal{H}^{n,q}(X,L)$ if and only if $\alpha$ is $\Box_{0}$-harmonic in the usual sense.
\begin{proof}
We only prove this isomorphism when $k=1$ and $E=\mathcal{O}_{X}$. The general case follows the same way.

Let $\|\cdot\|_{\varphi,\omega}$ be the $L^{2}$-norm defined by $\varphi$ and $\omega$. We use the de Rham--Weil isomorphism
\[
H^{n,q}(X,L\otimes\mathscr{I}(\varphi))\cong\frac{\mathrm{Ker}\bar\partial\cap L^{n,q}_{(2)}(X,L)_{\varphi,\omega}}{\mathrm{Im}\bar{\partial}}
\]
to represent a given cohomology class $[\alpha]\in H^{n,q}(X,L\otimes\mathscr{I}(\varphi))$ by a $\bar{\partial}$-closed $L$-valued $(n,q)$-form $\alpha$ with $\|\alpha\|_{\varphi,\omega}<\infty$. Since $\tilde{\omega}_{l}\geqslant\omega$, it is easy to verify that
\[
|\alpha_{Y}|^{2}_{\tilde{\omega}_{l}}e^{-\varphi_{\varepsilon}}dV_{\tilde{\omega}_{l}}\leqslant|\alpha|^{2}_{\omega}e^{-\varphi_{\varepsilon}}dV_{\omega},
\]
which leads to the inequality $\int_{Y}|\alpha_{Y}|^{2}_{\tilde{\omega}_{l}}e^{-\varphi_{\varepsilon}}\leqslant\int_{Y}|\alpha|^{2}_{\omega}e^{-\varphi}$. Then we have
$\|\alpha_{Y}\|_{\varphi_{\varepsilon},\tilde{\omega}_{l}}\leqslant\|\alpha\|_{\varphi,\omega}$ which implies
\[
\alpha_{Y}\in L^{n,q}_{(2)}(Y,L)_{\varphi_{\varepsilon},\tilde{\omega}_{l}}.
\]
By decomposition (\ref{e31}), we have a harmonic representative $\alpha_{\varepsilon,l}$ in
\[
\mathcal{H}^{n,q}_{\varphi_{\varepsilon,\tilde{\omega}_{l}}}(L),
\]
which means that $\Box_{\varepsilon,l}\alpha_{\varepsilon,l}=0$ on $Y$ for all $\varepsilon,l$. Moreover, since a harmonic representative minimizes the $L^{2}$-norm, we have
\[
   \|\alpha_{\varepsilon,l}\|_{\varphi_{\varepsilon},\tilde{\omega}_{l}}\leqslant\|\alpha_{Y}\|_{\varphi_{\varepsilon},\tilde{\omega}_{l}}\leqslant \|\alpha\|_{\varphi,\omega}.
\]
So we can take the limit $\tilde{\alpha}$ of (a subsequence of) $\{\alpha_{\varepsilon,l}\}$ such that
\[
\tilde{\alpha}\in[\alpha_{Y}].
\]
It is left to extend it to $X$.

Indeed, by the proof of Proposition \ref{p31}, (i), $\tilde{\alpha}$ maps to a $\bar{\partial}$-closed $L$-valued $(n-q,0)$-form $\ast\tilde{\alpha}$ on $X$. We denote this morphism by $S^{q}(\tilde{\alpha})$. Furthermore, it is shown by Proposition 2.2 in \cite{Wu17} that $\hat{\alpha}:=c_{n-q}\frac{\omega^{q}}{q!}\wedge S^{q}(\tilde{\alpha})$ is an $L$-valued $(n,q)$-form with
\[
  \hat{\alpha}|_{Y}=\tilde{\alpha}.
\]
Therefore we finally get an extension $\hat{\alpha}$ of $\tilde{\alpha}$. By definition,
\[
\hat{\alpha}\in\mathcal{H}^{n,q}(X,L\otimes\mathscr{I}(\varphi)).
\]
We denote this morphism by $i([\alpha])=\hat{\alpha}$.

On the other hand, for a given $\alpha\in\mathcal{H}^{n,q}(X,L\otimes\mathscr{I}(\varphi))$, by definition there exists an $\alpha_{\varepsilon,l}\in[\alpha_{Y}]$ with $\alpha_{\varepsilon,l}\in\mathcal{H}^{n,q}_{\varphi_{\varepsilon,\tilde{\omega}_{l}}}(L)$ for every $\varepsilon,l$. In particular, $\bar{\partial}\alpha_{\varepsilon,l}=0$. So all of the $\alpha_{\varepsilon,l}$ together with $\alpha_{Y}$ define a common cohomology class $[\alpha_{Y}]$ in $H^{n,q}(Y,L\otimes\mathscr{I}(\varphi))$. Here we use the property (c) that $\mathscr{I}(\varphi_{\varepsilon})=\mathscr{I}(\varphi)$ for every $\varepsilon$. It is left to extend this class to $X$.

We use the $S^{q}$ again. It maps $[\alpha_{Y}]$ to
\[
S^{q}(\alpha_{Y})\in H^{0}(X,\Omega^{n-q}_{X}\otimes L\otimes\mathscr{I}(\varphi)).
\]
Furthermore,
\[
   c_{n-q}\frac{\omega^{q}}{q!}\wedge S^{q}(\alpha_{Y})\in H^{n,q}(X,L\otimes\mathscr{I}(\varphi))
\]
with $[(c_{n-q}\frac{\omega^{q}}{q!}\wedge S^{q}(\alpha_{Y}))|_{Y}]=[\alpha_{Y}]$. Here we use the fact that $\omega$ is a K\"{a}hler metric. We denote this morphism by $j(\alpha)=[c_{n-q}\omega_{q}\wedge S^{q}(\alpha_{Y})]$. It is easy to verify that $i\circ j=\textrm{id}$ and $j\circ i=\textrm{id}$. The proof is finished.
\end{proof}
\end{proposition}

\subsection{A Koll\'{a}r-type injectivity theorem}
\label{sec:kollar}
We prove Theorem \ref{t12} to finish this section.
\begin{proof}[Proof of Theorem \ref{t12}]
Let $\varphi_{1}$ and $\varphi_{2}$ be the singular metrics on $L$ and $L\otimes H$ respectively mentioned in Sect.\ref{sec:asymptotic}, such that
\[
\mathscr{I}(\|L\|)=\mathscr{I}(\varphi_{1})
\]
and
\[
\mathscr{I}(\|L\otimes H\|)=\mathscr{I}(\varphi_{2}).
\]
In particular, it is easy to verify that there exits sections $u_{1},...,u_{m}\in L^{k_{1}}$ and $v_{1},...,v_{l}\in L^{k_{2}}\otimes H^{k_{2}}$ such that $(|u_{1}|^{2}+\cdots+|u_{m}|^{2})e^{-k_{1}\varphi_{1}}$ and $(|v_{1}|^{2}+\cdots+|v_{l}|^{2})e^{-k_{2}\varphi_{2}}$ are both bounded on $X$. Moreover, $(L,\varphi_{1})$ and $(L\otimes H,\varphi_{2})$ are pseudo-effective. Now we apply Proposition \ref{p32} to obtain that
\[
H^{q}(X,K_{X}\otimes L\otimes\mathscr{I}(\|L\|))\simeq H^{q}(X,K_{X}\otimes L\otimes\mathscr{I}(\varphi_{1}))\simeq\mathcal{H}^{n,q}(X,L\otimes\mathscr{I}(\varphi_{1}))
\]
and
\[
H^{q}(X,K_{X}\otimes L\otimes H\otimes\mathscr{I}(\|L\otimes H\|))\simeq H^{q}(X,K_{X}\otimes L\otimes H\otimes\mathscr{I}(\varphi_{2}))\simeq\mathcal{H}^{n,q}(X,L\otimes H\otimes\mathscr{I}(\varphi_{2})).
\]

It remains to prove that
\[
\otimes s:\mathcal{H}^{n,q}(X,L\otimes\mathscr{I}(\varphi_{1}))\rightarrow\mathcal{H}^{n,q}(X,L\otimes H\otimes\mathscr{I}(\varphi_{2}))
\]
is a well-defined morphism. If so, the injectivity is obvious. Let $\{\varphi_{\varepsilon,1}\}$ be the regularising sequence of $\varphi_{1}$, and let $\{\varphi_{\varepsilon,2}\}$ be the regularising sequence of $\varphi_{2}$. In particular, they are smooth on an open subvariety $Y$. Consider
\[
\alpha\in\mathcal{H}^{n,q}(X,L\otimes\mathscr{I}(\varphi_{1})).
\]
By definition, there exists a sequence $\{\alpha_{\varepsilon,l}\}$ such that $\Box_{\varepsilon,l}\alpha_{\varepsilon,l}=0$ and $\lim\alpha_{\varepsilon,l}=\alpha$ in the sense of $L^{2}$-topology. In particular, $\lim\partial_{\varphi_{\varepsilon,1}}^{\ast}\alpha_{\varepsilon,l}=0$ and
\[
\lim<i[\Theta_{L,\varphi_{\varepsilon,1}},\Lambda]\alpha_{\varepsilon,l},\alpha_{\varepsilon,l}>_{\varphi_{\varepsilon,1},\tilde{\omega}_{l}}=0
\]
as is shown in the proof of Proposition \ref{p31}. Now for an
\[
s\in H^{0}(X,H),
\]
we claim that $\Box_{0}(s\alpha)=0$. If so, we can verify that $s\alpha\in\mathcal{H}^{n,q}(X,L\otimes H\otimes\mathscr{I}(\|L\otimes H\|))$ as follows: since $s$ is a section of $H$, $s\in\mathfrak{a}(|H|)$. Thus, $s\alpha\in\mathcal{H}^{n,q}(X,L\otimes H\otimes\mathscr{I}(\|L\otimes H\|))$ by Proposition \ref{p21}, (ii).

Now we prove the claim. Observe that $\bar{\partial}(s\alpha)=0$,
\[
[s\alpha]\in H^{q}(X,K_{X}\otimes L\otimes H\otimes\mathscr{I}(\varphi_{2})).
\]
By Proposition \ref{p32}, there exists a sequence $\{\beta_{\varepsilon,l}\}$ on $Y$ such that $\Box_{\varepsilon,l}\beta_{\varepsilon,l}=0$ and $\beta_{\varepsilon,l}\in[s\alpha]$. It is left to prove that $\lim\beta_{\varepsilon,l}=(s\alpha)|_{Y}$ in the sense of $L^{2}$-topology. Indeed, since $\beta_{\varepsilon,l}\in[s\alpha]$, there exits an $(L\otimes H)$-valued $(n,q-1)$-form $\gamma_{\varepsilon,l}$ such that $s\alpha=\beta_{\varepsilon,l}+\bar{\partial}\gamma_{\varepsilon,l}$ on $Y$. From $\Box_{\varepsilon,l}\beta_{\varepsilon,l}=0$, we obtain $\bar{\partial}^{\ast}_{\varphi_{\varepsilon,2}}\beta_{\varepsilon,l}=0$. Now apply the formula (\ref{e32}) on $Y$, we get that
\[
\begin{split}
&\lim\|\bar{\partial}^{\ast}_{\varphi_{\varepsilon,2}}(s\alpha)\|^{2}_{\varphi_{\varepsilon,2},\tilde{\omega}_{l}}\\
=&\lim(\|\partial^{\ast}_{\varphi_{\varepsilon,2}}(s\alpha)\|^{2}_{\varphi_{\varepsilon,2},\tilde{\omega}_{l}}+<i[\Theta_{L\otimes H,\varphi_{\varepsilon,2}}, \Lambda](s\alpha),s\alpha>_{\varphi_{\varepsilon,2},\tilde{\omega}_{l}}).
\end{split}
\]
Since $\partial^{\ast}_{\varphi_{\varepsilon,2}}(s\alpha)=\ast\bar{\partial}\ast(s\alpha)=s\ast\bar{\partial}\ast\alpha=s\partial^{\ast}_{\varphi_{\varepsilon,1}}\alpha$,
\[
\begin{split}
\lim\|\partial^{\ast}_{\varphi_{\varepsilon,2}}(s\alpha)\|^{2}_{\varphi_{\varepsilon,2},\tilde{\omega}_{l}}&=\lim\|s\partial^{\ast}_{\varphi_{\varepsilon,1}}\alpha\|^{2}_{\varphi_{\varepsilon,2},\tilde{\omega}_{l}}\\
&\leqslant C\sup_{X}|s|^{2}e^{-\varphi_{3}}\lim\|\partial^{\ast}_{\varphi_{\varepsilon,1}}\alpha\|^{2}_{\varphi_{\varepsilon,1},\tilde{\omega}_{l}}\\
&=0.
\end{split}
\]
Here $\varphi_{3}$ is the singular metric on $H$ defined by $\mathfrak{a}(|H|)$. Now we explain the inequality. Fix $p\gg0$ and divisible enough that computes both of the multiplier ideals $\mathscr{I}(\|L\|)$ and $\mathscr{I}(\|L\otimes H\|)$. Let $\mu:\tilde{X}\rightarrow X$ be the smooth modification of $\mathfrak{a}(|H|)$, $\mathfrak{a}(|L^{p}|)$ and $\mathfrak{a}(|L^{p}\otimes H^{p}|)$, such that
\[
\mu^{\ast}\mathfrak{a}(|H|)=\mathcal{O}_{\tilde{X}}(-E), \mu^{\ast}\mathfrak{a}(|L^{p}|)=\mathcal{O}_{\tilde{X}}(-F)\textrm{ and }\mu^{\ast}\mathfrak{a}(|L^{p}\otimes H^{p}|)=\mathcal{O}_{\tilde{X}}(-G),
\]
where $E=\sum a_{i}E_{i}$, $F=\sum b_{i}E_{i}$ and $G=\sum c_{i}E_{i}$ have simple normal crossing support. Then for every $i$,
\[
c_{i}\leqslant pa_{i}+b_{i}
\]
and consequently
\[
\lfloor\frac{c_{i}}{p}\rfloor\leqslant a_{i}+\lfloor \frac{b_{i}}{p}\rfloor.
\]
Let $g_{i}$ be the local generator of $E_{i}$. Recall that the associated singular metrics are defined as follows:
\[
\varphi_{1}=\mu_{\ast}(\Pi_{i}\log|g_{i}|^{2\lfloor \frac{b_{i}}{p}\rfloor}), \varphi_{2}=\mu_{\ast}(\Pi_{i}\log|g_{i}|^{2\lfloor \frac{c_{i}}{p}\rfloor})\textrm{ and }\varphi_{3}=\mu_{\ast}(\Pi_{i}\log|g_{i}|^{2a_{i}}).
\]
Obviously, $\varphi_{1}+\varphi_{3}\leqslant\varphi_{2}+C$ for some constant $C$, which leads to the desired inequality. Observe that $\sup_{X}|s|^{2}e^{-\varphi_{3}}$ is bounded, the last equality follows.

In summary, we obtain that $\lim\partial^{\ast}_{\varphi_{\varepsilon,2}}(s\alpha)=0$. Similarly,
\[
\begin{split}
0\leqslant&\lim<i[\Theta_{L\otimes H,\varphi_{\varepsilon,2}},\Lambda](s\alpha),s\alpha>_{\varphi_{\varepsilon,2},\tilde{\omega}_{l}}\\
\leqslant&\sup_{X}|s|^{2}e^{-\varphi_{3}}\lim<i[\frac{1}{\delta}\Theta_{L,\varphi_{\varepsilon,1}},\Lambda]\alpha,\alpha>_{\varphi_{\varepsilon,1},\tilde{\omega}_{l}}\\
=&0.
\end{split}
\]
We obtain that $\lim<i[\Theta_{L\otimes H,\varphi_{\varepsilon,2}},\Lambda](s\alpha),s\alpha>_{\varphi_{\varepsilon,2},\tilde{\omega}_{l}}=0$. Therefore,
\[
\lim\|\bar{\partial}^{\ast}_{\varphi_{\varepsilon,2}}(s\alpha)\|^{2}_{\varphi_{\varepsilon,2},\tilde{\omega}_{l}}=0.
\]

Then we have
\[
\begin{split}
&\lim\|\bar{\partial}^{\ast}_{\varphi_{\varepsilon,2}}\bar{\partial}\gamma_{\varepsilon,l}\|^{2}_{\varphi_{\varepsilon,2},\tilde{\omega}_{l}}\\
=&\lim\|\bar{\partial}^{\ast}_{\varphi_{\varepsilon,2}}(s\alpha-\beta_{\varepsilon,l})\|^{2}_{\varphi_{\varepsilon,2},\tilde{\omega}_{l}}\\
=&0.
\end{split}
\]
In other words, $\lim\bar{\partial}^{\ast}_{\varphi_{\varepsilon,2}}\bar{\partial}\gamma_{\varepsilon,l}=0$. Hence
\[
\lim\|\bar{\partial}\gamma_{\varepsilon,l}\|^{2}_{\varphi_{\varepsilon,2},\tilde{\omega}_{l}}=\lim<\bar{\partial}^{\ast}_{\varphi_{\varepsilon,2}}\bar{\partial}\gamma_{\varepsilon,l},\gamma_{\varepsilon,l}>_{\varphi_{\varepsilon,2},\tilde{\omega}_{l}}=0.
\]
We conclude that $\lim\bar{\partial}\gamma_{\varepsilon,l}=0$. Equivalently, $\lim\beta_{\varepsilon,l}=s\alpha$ on $Y$. The proof is finished.
\end{proof}

One refers to \cite{Fuj12,Ko86a,Ko86b,Mat15a,Mat18} for a partial history of Koll\'{a}r's injectivity theorem.

\section{An asymptotic estimate}
\label{sec:estimate}
In this section we should prove Theorem \ref{t13}. The method is mainly borrowed from \cite{Mat14}. Recall there is the following lemma given in \cite{Mat14}.
\begin{lemma}[(=Lemma 4.3, \cite{Mat14})]
Let $X$ be a projective manifold of dimension $n$. Let $L$ (resp. $\mathcal{G}$) be a line bundle (resp. coherent sheaf) on $X$ and $\{\mathcal{I}_{k}\}^{\infty}_{k=1}$ be ideal sheaves on $X$ with the following assumption:

There exists a very ample line bundle $A$ on $X$ such that $H^{q}(X,A^{m}\otimes\mathcal{G}\otimes L^{k}\otimes\mathcal{I}_{k})=0$ for any $q>0$ and $k, m>0$.

Then for any $q\geqslant0$, we have
\[
h^{q}(X,\mathcal{G}\otimes L^{k}\otimes\mathcal{I}_{k})=O(k^{n-q})\textrm{ as }k\rightarrow\infty.
\]
\end{lemma}

Apart from this, we also need the following generalisation of Nadel's vanishing theorem.
\begin{theorem}\label{t41}
Let $X$ be a projective manifold of dimension $n$. Let $(E,H)$ be a Nakano (resp. semi-)positive \cite{Dem12} vector bundle on $X$, and let $(L,\varphi)$ be a (resp. big) pseudo-effective line bundle on $X$. Then
\[
H^{q}(X,K_{X}\otimes E\otimes L\otimes\mathscr{I}(\varphi))=0
\]
for $q>0$.
\begin{proof}
Let $\{\varphi_{\varepsilon}\}$ be the regularising sequence mentioned at the beginning of Sect.\ref{sec:harmonic}. Then
\[
\{H\otimes e^{-\varphi_{\varepsilon}}\}
\]
is a regularising approximation of $H\otimes e^{-\varphi}$. In particular, $H\otimes e^{-\varphi_{\varepsilon}}$ is smooth outside a closed subvariety $Z_{\varepsilon}$. Let $D^{1,0}_{\varepsilon}$ be the $(1,0)$-part of the Chern connection on $(E\otimes L)|_{X-Z_{\varepsilon}}$ associated with $H\otimes e^{-\varphi_{\varepsilon}}$. Following the same argument in Sect.\ref{sec:hodge}, we can easily generalise Proposition \ref{p32} (with the same notations there) as follows:
\[
\mathcal{H}^{n,q}(X,E\otimes L\otimes\mathscr{I}(\varphi))\simeq H^{n,q}(X,E\otimes L\otimes\mathscr{I}(\varphi)).
\]
Now for any $\alpha\in\mathcal{H}^{n,q}(X,E\otimes L\otimes\mathscr{I}(\varphi))$, we apply formula (\ref{e32}) with $\eta=1$, $\chi=\varphi_{\varepsilon}$ on $X-Z_{\varepsilon}$ to obtain that
\[
0=\|(D^{1,0}_{\varepsilon})^{\ast}\alpha\|^{2}_{\varphi_{\varepsilon}}+<i[\Theta_{E\otimes L,H\otimes e^{-\varphi_{\varepsilon}}},\Lambda]\alpha,\alpha>_{\varphi_{\varepsilon}}.
\]
Although formula (\ref{e32}) is formulated for a line bundle, we can arrange all the things for a higher rank vector bundle without obstacle. It is easy to verify that $<i[\Theta_{E\otimes L,H\otimes e^{-\varphi_{\varepsilon}}},\Lambda]\alpha,\alpha>_{\varphi_{\varepsilon}}>0$ if $\varepsilon\ll1$ and $\alpha\neq0$, which is a contradiction. Therefore we must have $\alpha=0$. Equivalently,
\[
H^{q}(X,K_{X}\otimes E\otimes L\otimes\mathscr{I}(\varphi))=0.
\]
\end{proof}
\end{theorem}

\begin{proof}[Proof of Theorem \ref{t13}]
Firstly, we prove it when $\mathcal{G}$ is locally free, i.e it is isomorphic to a vector bundle $E$. It is sufficient to show that there exists a very ample line bundle $A$ on $X$ (independent of $k$) such that
\[
H^{q}(X,A^{m}\otimes E\otimes L^{k}\otimes\mathscr{I}(\|L^{k}\|))=0
\]
for $q>0$ and $k, m>0$. By taking a sufficiently ample line bundle $A$ on $X$, we may assume that $A$ is very ample and $K_{X}^{-1}\otimes E\otimes A$ is positive in the sense of Nakano. Since $A$ is ample, there is a smooth metric $\psi$ on $A$ with strictly positive curvature.

Now let $\varphi_{k}$ be the singular metric on $L^{k}$ (see Sect.\ref{sec:asymptotic}) such that
\[
\mathscr{I}(\|L^{k}\|)=\mathscr{I}(\varphi_{k}).
\]
Then $(A^{m}\otimes L^{k},m\psi+\varphi_{k})$ is a big line bundle for all $m$ and $k$. By Theorem \ref{t41}, we obtain
\[
H^{q}(X,A^{m}\otimes E\otimes L^{k}\otimes\mathscr{I}(\|L^{k}\|))=H^{q}(X,K_{X}\otimes K^{-1}_{X}\otimes A^{m}\otimes E\otimes L^{k}\otimes\mathscr{I}(\varphi_{k}))=0.
\]
Here we use the fact that $\mathscr{I}(m\psi+\varphi_{k})=\mathscr{I}(\varphi_{k})$.

In the end, for a general $\mathcal{G}$, there exits a free resolution \cite{Har77}
\[
0\rightarrow E_{m}\rightarrow\cdots\rightarrow E_{1}\rightarrow\mathcal{G}\rightarrow0.
\]
We briefly explain the existence of such a resolution. Indeed, we can assume without loss of generality that $\mathcal{G}$ is globally generated by tensoring with $A$. Then the existence is elementary (see \cite{Kob87} for a suggestive argument). From this resolution, we eventually obtain the desired vanishing result for $\mathcal{G}$. The proof is finished.
\end{proof}

\section{Vanishing theorem}
\label{sec:applications}
Firstly, we prove Theorem \ref{t11}.
\begin{proof}[Proof of Theorem \ref{t11}]
All this really makes sense only if $\kappa(L)\geqslant0$, so we shall assume this.

Next, we prove the vanishing result by contradiction. Firstly, we claim that if
\[
H^{n,q}(X,L\otimes\mathscr{I}(\|L\|))
\]
is non-zero,
\[
h^{0}(L^{k-1})\leqslant\dim H^{n,q}(X,L^{k}\otimes\mathscr{I}(\|L^{k}\|)).
\]
In fact, let $\{s_{j}\}$ be a basis of $H^{0}(X,L^{k-1})$. Then for any
\[
\alpha\in H^{n,q}(X,L\otimes\mathscr{I}(\|L\|)),
\]
$\{s_{j}\alpha\}$ is linearly independent in $H^{n,q}(X,L^{k}\otimes\mathscr{I}(\|L^{k}\|))$ by Theorem \ref{t12}. Indeed, let $\varphi_{1}$ and $\varphi_{k}$ be the  metrics on $L$ and $L^{k}$ associated to $\mathscr{I}(\|L\|)$ and $\mathscr{I}(\|L^{k}\|)$ respectively. As a by-product of Proposition \ref{p21}, (iii), we have $\varphi_{k}=k\varphi_{1}$. So $i\Theta_{L^{k},\varphi_{k}}=ki\Theta_{L^{k},\varphi_{1}}$. Thus, Theorem \ref{t12} applies here, and it leads to the inequality.

Now suppose that $H^{n,q}(X,L\otimes\mathscr{I}(\|L\|))$ is non-zero for $q>n-\kappa(L)$. We have
\[
  h^{0}(L^{k-1})=h^{0}(L^{k-1}\otimes\mathscr{I}(\|L^{k-1}\|))\leqslant \dim H^{n,q}(X,L^{k}\otimes\mathscr{I}(\|L^{k}\|)).
\]
The first equality comes from the Proposition \ref{p21}, (i), and the second inequality is due to the claim. By the definition of Iitaka dimension \cite{Laz04}, we have
\[
  \limsup_{k\rightarrow\infty}\frac{h^{0}(L^{k-1})}{(k-1)^{\kappa(L)}}>0.
\]
It means that
\[
  \limsup_{k\rightarrow\infty}\frac{\dim H^{n,q}(X,L^{k}\otimes\mathscr{I}(\|L^{k}\|))}{(k-1)^{\kappa(L)}}>0.
\]
On the other hand, we have
\[
  \dim H^{n,q}(X,L^{k}\otimes\mathscr{I}(\|L^{k}\|))=O(k^{n-q})
\]
by Theorem \ref{t13}, so $n-q\geqslant\kappa(L)$. It contradicts to the fact that $q>n-\kappa(L)$. Hence
\[
H^{n,q}(X,L^{k}\otimes\mathscr{I}(\|L^{k}\|))=0
\]
for $q>n-\kappa(L)$.
\end{proof}

\section{Further discussion}
\label{sec:further}
In order to prove Theorem \ref{t15}, we need to extend Theorems \ref{t12} and \ref{t13} to the relative setting.

\subsection{The harmonic forms: local theory}
\label{sec:local}
Let $(X,\omega)$ be a compact K\"{a}hler manifold of dimension $n$, and let $L$ be a line bundle on $X$.

Firstly, we recall the harmonic theory in a local setting \cite{Tak95}. Let $V$ be a bounded domain with smooth boundary $\partial V$ on $X$. Moreover, there is a smooth plurisubharmonic exhaustion function $r$ of $V$ on $X$, with $\sup_{X}(|r|+|dr|)<\infty$. In particular, $V=\{r<0\}$ and $dr\neq0$ on $\partial V$. The volume form $dS$ of the real hypersurface $\partial V$ is defined by $dS:=\ast(dr)/|dr|_{\omega}$. Let $\varphi$ be a smooth Hermitian metric on $L$. Let $L^{p,q}_{(2)}(V,L)_{\varphi,\omega}$ be the space of $L$-valued $(p,q)$-forms on $V$ which are $L^{2}$-bounded with respect to $\varphi,\omega$. Setting $\tau:=dS/|dr|_{\omega}$ we define the inner product on $\partial V$ by
\[
[\alpha,\beta]_{\varphi}:=\int_{\partial V}<\alpha,\beta>_{\varphi}\tau
\]
for $\alpha,\beta\in L^{p,q}_{(2)}(V,L)_{\varphi,\omega}$. For a smooth $(p,q)$-form $\gamma$, let $e(\gamma)$ be the morphism $\gamma\wedge\cdot$. Then by Stokes' theorem we have the following:
\begin{equation}\label{e61}
\begin{split}
<\bar{\partial}\alpha,\beta>_{\varphi}&=<\alpha,\bar{\partial}^{\ast}\beta>_{\varphi}+[\alpha,e(\bar{\partial}r)^{\ast}\beta]_{\varphi}\\
<\partial_{\varphi}\alpha,\beta>_{\varphi}&=<\alpha,\partial^{\ast}_{\varphi}\beta>_{\varphi}+[\alpha,e(\partial r)^{\ast}\beta]_{\varphi}
\end{split}
\end{equation}
where $\bar{\partial}^{\ast},\partial^{\ast}_{\varphi}$ are the adjoint operators defined on $X$.

Now we furthermore assume that $i\Theta_{L,\varphi}\geqslant -C\omega$. Based on (\ref{e61}), if $e(\bar{\partial}r)^{\ast}\alpha=0$, it is proved in \cite{Tak95} that the Bochner formula on $V$ can be formulated as
\begin{equation}\label{e62}
\begin{split}
&\|\sqrt{\eta}(\bar{\partial}+e(\bar{\partial}\chi))\alpha\|^{2}_{\varphi,\omega}+\|\sqrt{\eta}\bar{\partial}^{\ast}\alpha\|^{2}_{\varphi,\omega}=\|\sqrt{\eta}(\partial^{\ast}_{\varphi}-e(\partial\chi)^{\ast})\alpha)\|^{2}_{\varphi,\omega}\\
&+\|\sqrt{\eta}\partial_{\varphi}\alpha\|^{2}_{\varphi,\omega}+<\eta i[\Theta_{L,\varphi}+\partial\bar{\partial}\chi,\Lambda]\alpha,\alpha>_{\varphi,\omega}.
\end{split}
\end{equation}
where $\eta$ is a positive smooth function on $X$ with $\chi:=\log\eta$.

We then define the space of harmonic forms on $V$ by
\[
\mathcal{H}^{n,q}_{\varphi}(V,L,r,\omega):=\{\alpha\in L^{n,q}_{(2)}(\bar{V},L)_{\varphi,\omega};\bar{\partial}\alpha=\bar{\partial}^{\ast}\alpha=e(\bar{\partial}r)^{\ast}\alpha=0\}.
\]

Now $(L,\varphi)$ is a pseudo-effective line bundle. Assume that there exits integers $k_{0}$, $m$ and sections $s_{1},...,s_{m}\in L^{k_{0}}$ such that
\[
(|s_{1}|^{2}+\cdots+|s_{m}|^{2})e^{-k_{0}\varphi}
\]
is bounded on $X$. Let $\{\varphi_{\varepsilon}\}$ be the regularising sequence given at the beginning of Sect.\ref{sec:harmonic}. Using the same notations there, the harmonic space with respect to $\varphi$ is defined as
\[
\begin{split}
\mathcal{H}^{n,q}(V,L\otimes\mathscr{I}(\varphi),r):=&\{\alpha\in L^{n,q}_{(2)}(V,L)_{\varphi,\omega};\textrm{there exits }\alpha_{l,\varepsilon}\in[\alpha]\textrm{ such that }\\
&\alpha_{l,\varepsilon}\in\mathcal{H}^{n,q}_{\varphi_{\varepsilon}}(V-Z,L,r,\tilde{\omega}_{l})\textrm{ and }\alpha_{l,\varepsilon}\rightarrow\alpha\textrm{ in }L^{2}\textrm{-limit}\}.
\end{split}
\]
We then generalise the work in \cite{Tak95} here.
\begin{proposition}\label{p61}
We have the following conclusions:
\begin{enumerate}
  \item Assume $\alpha\in L^{n,q}_{(2)}(X,L)_{\varphi,\omega}$ satisfied $e(\bar{\partial}r)^{\ast}\alpha=0$ on $V$. Then $\alpha$ satisfies $\bar{\partial}\alpha=\lim\bar{\partial}^{\ast}_{\varphi_{\varepsilon}}\alpha=0$ on $V$ if and only if $\bar{\partial}\ast\alpha=0$ and $\lim<ie(\Theta_{L,\varphi_{\varepsilon}}+\partial\bar{\partial}r)\Lambda\alpha,\alpha>_{\varphi_{\varepsilon}}=0$ on $V$.
  \item $\mathcal{H}^{n,q}(V,L\otimes\mathscr{I}(\varphi),r)$ is independent of the choice of exhaustion function $r$.
  \item $\mathcal{H}^{n,q}(V,L\otimes\mathscr{I}(\varphi),r)\simeq H^{q}(V,K_{V}\otimes L\otimes\mathscr{I}(\varphi))$.
  \item For Stein open subsets $V_{1},V_{2}$ in $V$ such that $V_{2}\subset V_{1}$, the restriction map
  \[
  \mathcal{H}^{n,q}(V_{1},\mathscr{I}(\varphi),r)\rightarrow\mathcal{H}^{n,q}(V_{2},\mathscr{I}(\varphi),r)
  \]
  is well-defined, and further it satisfies the following commutative diagram:
 \begin{equation*}
  \xymatrix{
  \mathcal{H}^{n,q}(V_{1},\mathscr{I}(\varphi),r) \ar[d]_{i^{1}_{2}} \ar[r]^{S^{q}_{V_{1}}} & H^{0}(V_{1},\Omega^{n-q}_{Y}\otimes \mathscr{I}(\varphi)) \ar[d] \\
  \mathcal{H}^{n,q}(V_{2},\mathscr{I}(\varphi),r) \ar[r]^{S^{q}_{V_{2}}} & H^{0}(V_{2},\Omega^{n-q}_{Y}\otimes \mathscr{I}(\varphi)).
  }
 \end{equation*}

\end{enumerate}
\begin{proof}
The proof uses the same argument as Theorems 4.3 and 5.2 in \cite{Tak95} with minor adjustment. So we only provide the necessary details.

(i) Let $\psi=\varphi+r$ and $\psi_{\varepsilon}=\varphi_{\varepsilon}+r$. If $\bar{\partial}\alpha=\lim\bar{\partial}^{\ast}_{\varphi_{\varepsilon}}\alpha=0$, then $\lim\bar{\partial}^{\ast}_{\psi_{\varepsilon}}\alpha=0$ and so $\lim\Box_{\psi_{\varepsilon}}\alpha=0$. By formula (\ref{e62}) we obtain
\[
\lim(\|\partial^{\ast}_{\psi_{\varepsilon}}\alpha\|^{2}_{\psi_{\varepsilon}}+<ie(\Theta_{L,\varphi_{\varepsilon}}+\partial\bar{\partial}r)\Lambda\alpha,\alpha>_{\psi_{\varepsilon}})=0
\]
on $V$. Since $<ie(\Theta_{L,\varphi_{\varepsilon}})\Lambda\alpha,\alpha>_{\psi_{\varepsilon}}\geqslant-\varepsilon\omega$ and
\[
[ie(\partial\bar{\partial}r)\Lambda\alpha,\alpha]_{\psi_{\varepsilon}}\geqslant0,
\]
the equality above implies that
\[
\begin{split}
&\ast\bar{\partial}\ast\alpha=0\textrm{ and }\\
&\lim<ie(\Theta_{L,\varphi_{\varepsilon}})\Lambda\alpha,\alpha>_{\psi_{\varepsilon}}=\lim[ie(\partial\bar{\partial}r)\Lambda\alpha,\alpha]_{\psi_{\varepsilon}}=0.
\end{split}
\]
Equivalently,
\[
\begin{split}
&\bar{\partial}\ast\alpha=0\textrm{ and }\\
&\lim<ie(\Theta_{L,\varphi_{\varepsilon}}+\partial\bar{\partial}r)\Lambda\alpha,\alpha>_{\varphi_{\varepsilon}}=0.
\end{split}
\]
The necessity is proved.

Now assume that $\bar{\partial}\ast\alpha=0$ and $\lim<ie(\Theta_{L,\varphi_{\varepsilon}}+\partial\bar{\partial}r)\Lambda\alpha,\alpha>_{\varphi_{\varepsilon}}=0$. Since $r$ is plurisubharmonic and $\lim<ie(\Theta_{L,\varphi_{\varepsilon}})\Lambda\alpha,\alpha>_{\varphi_{\varepsilon}}\geqslant0$, we have
\[
\lim<ie(\partial\bar{\partial}r)\Lambda\alpha,\alpha>_{\varphi_{\varepsilon}}=\lim<ie(\Theta_{L,\varphi_{\varepsilon}})\Lambda\alpha,\alpha>_{\varphi_{\varepsilon}}=0.
\]
By formula (\ref{e62}) we have $\bar{\partial}\alpha=\lim\bar{\partial}^{\ast}_{\varphi_{\varepsilon}}\alpha=0$.

(ii) Let $\tau$ be an arbitrary smooth plurisubharmonic function on $V$. Donnelly and Xavier's formula \cite{DoX84} implies that $\bar{\partial}e(\bar{\partial}\tau)^{\ast}\alpha=ie(\partial\bar{\partial}\tau)\Lambda\alpha$ if
\[
\alpha\in\mathcal{H}^{n,q}(V,L\otimes\mathscr{I}(\varphi),r).
\]
Therefore
\[
\begin{split}
<ie(\partial\bar{\partial}\tau)\Lambda\alpha,\alpha>_{\varphi_{\varepsilon}-\tau}&=<\bar{\partial}e(\bar{\partial}\tau)^{\ast}\alpha,\alpha>_{\varphi_{\varepsilon}-\tau}\\
&=<e(\bar{\partial}\tau)^{\ast}\alpha,\bar{\partial}^{\ast}_{\varphi_{\varepsilon}-\tau}\alpha>_{\varphi_{\varepsilon}-\tau}\\
&=<e(\bar{\partial}\tau)^{\ast}\alpha,\bar{\partial}^{\ast}_{\varphi_{\varepsilon}}\alpha>_{\varphi_{\varepsilon}-\tau}-\|e(\bar{\partial}\tau)^{\ast}\alpha\|^{2}_{\varphi_{\varepsilon}-\tau}.
\end{split}
\]
Here we use $\bar{\partial}^{\ast}_{\varphi_{\varepsilon}-\tau},\bar{\partial}^{\ast}_{\varphi_{\varepsilon}}$ to denote the adjoint operators with respect to $\varphi_{\varepsilon}-\tau$ and $\varphi_{\varepsilon}$. Take the limit with respect to $\varepsilon$, we then obtain that
\[
<ie(\partial\bar{\partial}\tau)\Lambda\alpha,\alpha>_{\varphi-\tau}=-\|e(\bar{\partial}\tau)^{\ast}\alpha\|^{2}_{\varphi-\tau}.
\]
Notice that $\tau$ is plurisubharmonic, we actually have
\[
<ie(\partial\bar{\partial}\tau)\Lambda\alpha,\alpha>_{\varphi-\tau}=\|e(\bar{\partial}\tau)^{\ast}\alpha\|^{2}_{\varphi-\tau}=0.
\]
Combine with (i), we eventually obtain that
\[
\mathcal{H}^{n,q}(V,L\otimes\mathscr{I}(\varphi),r)=\mathcal{H}^{n,q}(V,L\otimes\mathscr{I}(\varphi),r+\tau)
\]
for any smooth plurisubharmonic $\tau$, hence the desired conclusion.

(iii) When $\varphi$ is smooth, it is proved in \cite{Tak95}, Theorem 4.5, (b). When $\varphi$ is singular, we could apply Theorem 4.5, (b) in \cite{Tak95} to its regularising sequence to obtain the desired conclusion. This approximation argument is similar with Proposition \ref{p32}, and we omit the details here.

(iv) is intuitive due to the discussions in the global setting. In particular, $S^{q}_{V_{i}}$ with $i=1,2$ are similarly defined as in the proof of Proposition \ref{p32}.
\end{proof}
\end{proposition}

In the rest part of this paper, we are always working on the setting in Theorem \ref{t15}. Namely, let $f:X\rightarrow Y$ be a surjective morphism between projective manifolds, and let $L$ be a pseudo-effective line bundle on $X$. Let $l$ be the dimension of a general fibre $F$ of $f$.

Let $\{U,r_{U}\}$ be a finite Stein covering of $Y$ with smooth strictly plurisubharmonic exhaustion function $r_{U}$. Let
\[
\mathcal{H}^{n,q}(f^{-1}(U),L\otimes\mathscr{I}(\varphi),f^{\ast}r_{U})
\]
be the harmonic space defined above. Then the data
\[
\{\mathcal{H}^{n,q}(f^{-1}(U),L\otimes\mathscr{I}(\varphi),f^{\ast}r_{U}),i^{1}_{2}\}
\]
with the restriction morphisms
\[
i^{1}_{2}:\mathcal{H}^{n,q}(f^{-1}(U_{1}),L\otimes\mathscr{I}(\varphi),f^{\ast}r_{U_{1}})\rightarrow\mathcal{H}^{n,q}(f^{-1}(U_{2}),L\otimes\mathscr{I}(\varphi),f^{\ast}r_{U_{2}}),
\]
$(U_{2},r_{U_{2}})\subset(U_{1},r_{U_{1}})$, yields a presheaf \cite{Har77} on $Y$ by Proposition \ref{p61}, (iv). We denote the associated sheaf by $f_{\ast}\mathcal{H}^{n,q}(L\otimes\mathscr{I}(\varphi))$. Since
\[
R^{q}f_{\ast}(K_{X}\otimes L\otimes\mathscr{I}(\varphi))
\]
is defined as the sheaf associated with the presheaf
\[
U\rightarrow H^{q}(f^{-1}(U),K_{X}\otimes L\otimes\mathscr{I}(\varphi)),
\]
the sheaf $f_{\ast}\mathcal{H}^{n,q}(L\otimes\mathscr{I}(\varphi))$ is isomorphic to $R^{q}f_{\ast}(K_{X}\otimes L\otimes\mathscr{I}(\varphi))$ by combing with Proposition \ref{p61}, (iii) and (the proof of) Theorem 5.2, (i) in \cite{Tak95}. Moreover, the whole argument is even valid for a collection of local singular metrics $\{f^{-1}(U),\varphi_{U}\}$ on $L$ associated to $\mathscr{I}(f,\|L\|)$ (see Sect.\ref{sec:asymptotic}). Let $\varphi$ denote the collection of $\{\varphi_{U}\}$ by abusing the notation. Remember that $\mathscr{I}(\varphi)$ is globally defined.

Then Proposition \ref{p32} is generalised as follows.
\begin{proposition}\label{p62}
Let $\varphi$ be the associated metric (see Sect.\ref{sec:asymptotic}) of $\mathscr{I}(f,\|L\|)$. Then
\[
R^{q}f_{\ast}(K_{X}\otimes L\otimes\mathscr{I}(\varphi))\simeq f_{\ast}\mathcal{H}^{n,q}(L\otimes\mathscr{I}(\varphi)).
\]
\end{proposition}

\subsection{Injectivity theorem}
In this subsection, we should extend Theorem \ref{t12}.
\begin{theorem}\label{t61}
Let $L$ be a line bundle on $X$ with $\kappa(L,f)\geqslant0$. For a (non-zero) section $s$ of $L$, the multiplication map induced by the tensor product with $s$
\[
\Phi:R^{q}f_{\ast}(K_{X}\otimes L\otimes\mathscr{I}(f,\|L\|))\rightarrow R^{q}f_{\ast}(K_{X}\otimes L^{2}\otimes\mathscr{I}(f,\|L^{2}\|))
\]
is well-defined and injective for any $q\geqslant0$. In particular, $R^{q}f_{\ast}(K_{X}\otimes L\otimes\mathscr{I}(f,\|L\|))$ is torsion-free for every $q$.
\begin{proof}
Let $\{U,r_{U}\}$ be a finite Stein covering of $Y$ with smooth strictly plurisubharmonic exhaustion function $r_{U}$. From the discussion in Sect.\ref{sec:asymptotic}, there is a collection of (local) singular metrics $\varphi_{1}=\{f^{-1}(U),\varphi_{U,1}\}$ on $L$ and $\varphi_{2}=\{f^{-1}(U),\varphi_{U,2}\}$ on $L^{2}$ such that
\[
\mathscr{I}(f,\|L\|)=\mathscr{I}(\varphi_{1})
\]
and
\[
\mathscr{I}(f,\|L^{2}\|)=\mathscr{I}(\varphi_{2})
\]
respectively. In particular, it is a by-product of Proposition \ref{p21}, (vi), that $\varphi_{2}=2\varphi_{1}$, namely $\varphi_{U,2}=2\varphi_{U,1}$ for every $U$. Then in the view of Proposition \ref{p62} it is left to prove that
\[
f_{\ast}\mathcal{H}^{n,q}(L\otimes\mathscr{I}(\varphi_{1}))\rightarrow f_{\ast}\mathcal{H}^{n,q}(L^{2}\otimes\mathscr{I}(\varphi_{2}))
\]
is well-defined and injective.

Let $\alpha\in\mathcal{H}^{n,q}(f^{-1}(U),L\otimes\mathscr{I}(\varphi_{1}),f^{\ast}r_{U})$, and let $\{\varphi_{\varepsilon,1}\}$ be a regularising sequence of $\varphi_{1}$. Certainly this regularising sequence is interpreted that $\varphi_{\varepsilon,1}=\{f^{-1}(U),\varphi_{U,\varepsilon,1}\}$ and $\{\varphi_{U,\varepsilon,1}\}$ is a regularising sequence of $\varphi_{U,1}$ as in Sect.\ref{sec:forms}. Obviously $\{2\varphi_{\varepsilon,1}\}$ is a regularising sequence of $\varphi_{2}$. By definition, there exists a sequence $\{\alpha_{\varepsilon,l}\}$ such that $\alpha_{\varepsilon,l}\in\mathcal{H}^{n,q}_{\varphi_{\varepsilon}}(f^{-1}(U)-Z,L)$ and $\lim\alpha_{\varepsilon,l}=\alpha$ in the sense of $L^{2}$-topology. Apply formula (\ref{e62}) to $\alpha_{\varepsilon,l}$ on $f^{-1}(U)-Z$ and remember that $e(\partial r_{U})^{\ast}\alpha_{\varepsilon,l}=0$, we obtain
\[
\begin{split}
0=&\|\bar{\partial}\alpha_{l,\varepsilon}\|^{2}_{\varphi_{\varepsilon,1},\tilde{\omega}_{l}}+\|\bar{\partial}^{\ast}_{\varphi_{\varepsilon,1}}\alpha_{l,\varepsilon}\|^{2}_{\varphi_{\varepsilon,1},\tilde{\omega}_{l}}\\
=&\|\partial^{\ast}_{\varphi_{\varepsilon,1}}\alpha_{l,\varepsilon}\|^{2}_{\varphi_{\varepsilon,1},\tilde{\omega}_{l}}+<i[\Theta_{L,\varphi_{\varepsilon,1}},\Lambda]\alpha_{l,\varepsilon},\alpha_{l,\varepsilon}>_{\varphi_{\varepsilon,1},\tilde{\omega}_{l}}.
\end{split}
\]

Remember that $i\Theta_{L,\varphi_{U,1}}\geqslant0$. Thus, $\lim\partial_{\varphi_{\varepsilon,1}}^{\ast}\alpha_{\varepsilon,l}=0$ and
\[
\lim<i[\Theta_{L,\varphi_{\varepsilon,1}},\Lambda]\alpha_{\varepsilon,l},\alpha_{\varepsilon,l}>_{\varphi_{\varepsilon,1},\tilde{\omega}_{l}}=0.
\]

Now for an
\[
s\in H^{0}(X,L),
\]
we have $\bar{\partial}(s\alpha)=0$. Let $\mathfrak{a}$ be the base-ideal of $|L|$ relative to $f$, so $s\in\mathfrak{a}$. Then
\[
[s\alpha]\in H^{q}(f^{-1}(U),K_{X}\otimes L^{2}\otimes\mathscr{I}(\varphi_{2}))
\]
by Proposition \ref{p21}, (v). By Proposition \ref{p62}, there exists a sequence $\{\beta_{\varepsilon,l}\}$ on $f^{-1}(U)-Z$ such that $\Box_{\varepsilon,l}\beta_{\varepsilon,l}=0$ and $\beta_{\varepsilon,l}\in[s\alpha]$. It is left to prove that $\lim\beta_{\varepsilon,l}=(s\alpha)|_{f^{-1}(U)-Z}$ in the sense of $L^{2}$-topology. Indeed, since $\beta_{\varepsilon,l}\in[s\alpha]$, there exits an $L^{2}$-valued $(n,q-1)$-form $\gamma_{\varepsilon,l}$ such that $s\alpha=\beta_{\varepsilon,l}+\bar{\partial}\gamma_{\varepsilon,l}$ on $f^{-1}(U)-Z$. Since $\Box_{\varepsilon,l}\beta_{\varepsilon,l}=0$, $\bar{\partial}^{\ast}_{2\varphi_{\varepsilon,1}}\beta_{\varepsilon,l}=0$. Now apply the formula (\ref{e62}) on $f^{-1}(U)-Z$, we obtain that
\[
\begin{split}
&\lim\|\bar{\partial}^{\ast}_{\varphi_{2\varepsilon,1}}(s\alpha)\|^{2}_{2\varphi_{\varepsilon,1},\tilde{\omega}_{l}}\\
=&\lim(\|\partial^{\ast}_{2\varphi_{\varepsilon,1}}(s\alpha)\|^{2}_{2\varphi_{\varepsilon,1},\tilde{\omega}_{l}}+<i[\Theta_{L^{2},2\varphi_{\varepsilon,1}}, \Lambda](s\alpha),s\alpha>_{2\varphi_{\varepsilon,1},\tilde{\omega}_{l}}).
\end{split}
\]
Since $\partial^{\ast}_{2\varphi_{\varepsilon,1}}(s\alpha)=\ast\bar{\partial}\ast(s\alpha)=s\ast\bar{\partial}\ast\alpha=s\partial^{\ast}_{\varphi_{\varepsilon,1}}\alpha$,
\[
\begin{split}
\lim\|\partial^{\ast}_{2\varphi_{\varepsilon,1}}(s\alpha)\|^{2}_{2\varphi_{\varepsilon,1},\tilde{\omega}_{l}}&=\lim\|s\partial^{\ast}_{\varphi_{\varepsilon,1}}\alpha\|^{2}_{2\varphi_{\varepsilon,1},\tilde{\omega}_{l}}\\
&\leqslant\sup_{X}|s|^{2}e^{-\varphi_{3}}\lim\|\partial^{\ast}_{\varphi_{\varepsilon,1}}\alpha\|^{2}_{\varphi_{\varepsilon,1},\tilde{\omega}_{l}}\\
&=0.
\end{split}
\]
Here $\varphi_{3}$ is the singular metric on $L|_{f^{-1}(U)}$ defined by $\mathfrak{a}(f,|L|)$, and the inequality has been explained in the proof of Theorem \ref{t12}. It essentially follows from the fact that
\[
\mathfrak{a}(f,|L|)\cdot\mathscr{I}(f,\|L\|)\subseteq\mathscr{I}(f,\|L^{2}\|).
\]
Since $\sup_{X}|s|^{2}e^{-\varphi_{3}}$ is obviously bounded, we obtain that $\lim\partial^{\ast}_{2\varphi_{\varepsilon,1}}(s\alpha)=0$. Moreover,
\[
\begin{split}
0\leqslant&\lim<i[\Theta_{L^{2},2\varphi_{\varepsilon,1}},\Lambda](s\alpha),s\alpha>_{2\varphi_{\varepsilon,1},\tilde{\omega}_{l}}\\
\leqslant&\sup_{X}|s|^{2}e^{-\varphi_{3}}\lim<i[2\Theta_{L,\varphi_{\varepsilon,1}},\Lambda]\alpha,\alpha>_{\varphi_{\varepsilon,1},\tilde{\omega}_{l}}\\
=&0.
\end{split}
\]
We obtain that $\lim<i[\Theta_{L^{2},2\varphi_{\varepsilon,1}},\Lambda](s\alpha),s\alpha>_{2\varphi_{\varepsilon,1},\tilde{\omega}_{l}}=0$. Therefore,
\[
\lim\|\bar{\partial}^{\ast}_{2\varphi_{\varepsilon,1}}(s\alpha)\|^{2}_{2\varphi_{\varepsilon,1},\tilde{\omega}_{l}}=0.
\]

Then we have
\[
\begin{split}
&\lim\|\bar{\partial}^{\ast}_{2\varphi_{\varepsilon,1}}\bar{\partial}\gamma_{\varepsilon,l}\|^{2}_{2\varphi_{\varepsilon,1},\tilde{\omega}_{l}}\\
=&\lim\|\bar{\partial}^{\ast}_{2\varphi_{\varepsilon,1}}(s\alpha-\beta_{\varepsilon,l})\|^{2}_{2\varphi_{\varepsilon,1},\tilde{\omega}_{l}}\\
=&0.
\end{split}
\]
In other words, $\lim\bar{\partial}^{\ast}_{\varphi_{\varepsilon,2}}\bar{\partial}\gamma_{\varepsilon,l}=0$. Hence
\[
\lim\|\bar{\partial}\gamma_{\varepsilon,l}\|^{2}_{2\varphi_{\varepsilon,1},\tilde{\omega}_{l}}=\lim<\bar{\partial}^{\ast}_{2\varphi_{\varepsilon,1}}\bar{\partial}\gamma_{\varepsilon,l},\gamma_{\varepsilon,l}>_{2\varphi_{\varepsilon,1},\tilde{\omega}_{l}}=0.
\]
We conclude that $\lim\bar{\partial}\gamma_{\varepsilon,l}=0$. Equivalently, $\lim\beta_{\varepsilon,l}=s\alpha$ on $f^{-1}(U)-Z$. In summary,
\[
s\alpha\in\mathcal{H}^{n,q}(f^{-1}(U),L^{2}\otimes\mathscr{I}(\varphi_{2})).
\]
Then we have successfully proved that
\[
f_{\ast}\mathcal{H}^{n,q}(L\otimes\mathscr{I}(\varphi_{1}))\rightarrow f_{\ast}\mathcal{H}^{n,q}(L^{2}\otimes\mathscr{I}(\varphi_{2}))
\]
is well-defined. The injectivity is obvious.
\end{proof}
\end{theorem}

\subsection{Asymptotic estimate and vanishing theorem}
\label{sec:relative}
In this subsection, we should extend Theorem \ref{t13}.

\begin{theorem}\label{t62}
Let $L$ be a pseudo-effective line bundle on $X$. Then for any coherent sheaf $\mathcal{G}$ and $q\geqslant0$, we have
\[
    \textrm{rank }R^{q}f_{\ast}(L^{k}\otimes\mathcal{G}\otimes\mathscr{I}(\|L^{k}\|))=O(k^{l-q}).
\]
\begin{proof}
Apply Theorem \ref{t13} on the general fibre, we then obtain the desired result.
\end{proof}
\end{theorem}

In the end, we prove Theorem \ref{t15}.

\begin{proof}[Proof of Theorem \ref{t15}]
It is trivial when $\kappa(L,f)=-\infty$.

When $\kappa(L,f)\geqslant0$, we use the same argument as before. Firstly, we claim that if
\[
R^{q}f_{\ast}(X,K_{X}\otimes L\otimes\mathscr{I}(\|L\|))
\]
is non-zero,
\[
\textrm{rank }f_{\ast}L^{k-1}\leqslant\textrm{rank }R^{q}f_{\ast}(X,K_{X}\otimes L^{k}\otimes\mathscr{I}(\|L^{k}\|)).
\]
In fact, let $\{s_{j}\}$ be a local basis of $f_{\ast}L^{k-1}$. Then for any local section
\[
\alpha\in R^{q}f_{\ast}(X,K_{X}\otimes L\otimes\mathscr{I}(\|L\|)),
\]
$\{s_{j}\alpha\}$ is linearly independent in $R^{q}f_{\ast}(X,K_{X}\otimes L^{k}\otimes\mathscr{I}(\|L^{k}\|))$ by Theorem \ref{t61}. It leads to the inequality.

Now suppose that $R^{q}f_{\ast}(X,K_{X}\otimes L\otimes\mathscr{I}(\|L\|))$ is non-zero for $q>l-\kappa(L,f)$. We have
\[
  \textrm{rank }f_{\ast}L^{k-1}=\textrm{rank }f_{\ast}(L^{k-1}\otimes\mathscr{I}(\|L^{k-1}\|))\leqslant\textrm{rank }R^{q}f_{\ast}(X,K_{X}\otimes L^{k}\otimes\mathscr{I}(\|L^{k}\|)).
\]
The first equality comes from the Proposition \ref{p21}, (iv), and the second inequality is due to the claim. By the definition of relative Iitaka dimension (see Sect.\ref{sec:iitaka}), we have
\[
  \limsup_{k\rightarrow\infty}\frac{\textrm{rank }f_{\ast}L^{k-1}}{(k-1)^{\kappa(L,f)}}>0.
\]
It means that
\[
  \limsup_{k\rightarrow\infty}\frac{\textrm{rank }R^{q}f_{\ast}(X,K_{X}\otimes L^{k}\otimes\mathscr{I}(\|L^{k}\|))}{(k-1)^{\kappa(L,f)}}>0.
\]
On the other hand, we have
\[
  \textrm{rank }R^{q}f_{\ast}(L^{k}\otimes\mathcal{G}\otimes\mathscr{I}(\|L^{k}\|))=O(k^{l-q})
\]
by Theorem \ref{t62}, so $l-q\geqslant\kappa(L,f)$. It contradicts to the fact that $q>l-\kappa(L,f)$. Hence
\[
R^{q}f_{\ast}(X,K_{X}\otimes L\otimes\mathscr{I}(\|L\|))=0
\]
for $q>l-\kappa(L,f)$.
\end{proof}

\begin{acknowledgements}
The author wants to thank Prof. Jixiang Fu for his suggestion and encouragement.
\end{acknowledgements}

\end{document}